\documentclass[journal,12pt,onecolumn,draftclsnofoot]{IEEEtran}
\usepackage{cite,mathtools,amssymb,amsfonts,amsthm,mathrsfs,algorithm,algorithmic,graphicx,textcomp}
\usepackage[hidelinks]{hyperref}
\usepackage{pgfplots}
   \pgfplotsset{compat=newest}
     \usetikzlibrary{external}
\newtheorem{theorem}{Theorem}
\newtheorem{proposition}{Proposition}

\newtheorem{definition}{Definition}

\newtheorem{lemma}{Lemma}
\newcounter{experiment}[section]
\newenvironment{experiment}[1][]{\refstepcounter{experiment}\par
   \noindent \textbf{Experiment \theexperiment: #1} \rmfamily}{\par}
\newcommand{\norm}[1]{\left\lVert#1\right\rVert}

\DeclareMathOperator{\vspan}{span}
\allowdisplaybreaks

\begin{document}
\title{Dynamic Mode Decomposition of Control-Affine Nonlinear Systems using Discrete Control Liouville Operators}
\author{Zachary Morrison, Moad Abudia, Joel A. Rosenfeld, and Rushikesh Kamalapurkar
\thanks{This research was supported by the Air Force Office of Scientific Research (AFOSR) under contract number FA9550-20-1-0127 and the National Science Foundation (NSF) under awards number 2027999. Any opinions, findings and conclusions or recommendations expressed in this material are those of the author(s) and do not necessarily reflect the views of the sponsoring agencies.}
\thanks{Z. Morrison, Moad Abudia, and R. Kamalapurkar are with the School of Mechanical and Aerospace Engineering, Oklahoma State University, Stillwater, OK, 74074, United States of America (e-mail: {zachmor,abudia,rushikesh.kamalapurkar}@okstate.edu). J. A. Rosenfeld is with the Department of Mathematics and Statistics, University of South Florida, Tampa, FL, 33620, United States of America (e-mail:rosenfeldj@usf.edu).}\vspace{-1.5em}}

\maketitle
\thispagestyle{empty}
\begin{abstract}
The representation of a nonlinear dynamical system as infinite-dimensional linear operator over a Hilbert space of functions enables the study of the nonlinear system via pseudo-spectral analysis of the corresponding operator. In this paper, we develop a novel operator representation of discrete-time, control-affine nonlinear dynamical systems. We also demonstrate that this representation can be used to predict the behavior of the closed-loop system in response to a given feedback law. The representation is learned using recorded snapshots of the system state resulting from arbitrary, potentially open-loop control inputs. We thereby extend the predictive capabilities of dynamic mode decomposition to discrete-time nonlinear systems that are affine in control. We validate the method using two numerical experiments by predicting the response of a controlled Duffing oscillator to a known feedback law. The advantages of the developed method relative to existing techniques in the literature are also demonstrated.
\end{abstract}

\begin{IEEEkeywords}
Dynamic mode decomposition with control, composition operators, reproducing kernel Hilbert spaces, nonlinear system ID
\end{IEEEkeywords}
\section{Introduction}
\IEEEPARstart{I}{n} this paper, a novel representation of discrete-time control-affine nonlinear systems as infinite-dimensional linear operators over reproducing kernel Hilbert spaces (RKHSs) is introduced. This effort is inspired by the method first developed in \cite{SCC.Rosenfeld.Kamalapurkar2021}, which introduced similar operator representations for \textit{continuous-time} dynamical systems. The idea of representing a nonlinear system as an infinite-dimensional linear operator in Hilbert space was first put forth by B.O. Koopman in \cite{SCC.Koopman1931} and the resulting composition operator is aptly known as the Koopman operator. This higher-dimensional space is typically referred to as the \textit{feature space} or \textit{lifted space} and the Koopman operator acts as a composition operator on the lifted space. In recent years, dynamic mode decomposition (DMD) and other data-driven methods have seen a resurgence due to the abundance of data and increased the availability of computational power \cite{SCC.Kutz.Brunton.ea2016}. An example of the application of DMD can be seen in the fluid mechanics community, where modal decomposition of fluid flows is accomplished \cite{SCC.Schmid2010}, \cite{SCC.Mezic2013}. In a more general sense, DMD is intimately connected to the Koopman operator as DMD is one method used to approximate the Koopman operator associated with the dynamical system \cite{SCC.Mezic2013}, \cite{SCC.Tu.Rowley.ea2014a}. 

The Koopman approach is amenable to spectral methods in linear operator theory in certain cases, e.g. see \cite{SCC.Gonzalez.Abudia.ea2021}, but spectral convergence cannot be guaranteed in general; therefore, Koopman DMD methods are pseudo-spectral numerical methods. Despite this theoretical limitation, Koopman DMD and Liouville DMD methods in both continuous and discrete time have been shown to exhibit remarkable predictive accuracy over finite-time horizons \cite{SCC.Kutz.Brunton.ea2016}. Moreover, Koopman DMD allows one to study dynamical systems without direct knowledge of the dynamics, as Koopman DMD is strictly data driven and requires no knowledge of the dynamical system \cite{SCC.Kutz.Brunton.ea2016}. For measurements corrupted by noise or in the case of stochastic systems, robust approximations of the Koopman operator can be formulated  \cite{SCC.Sharma.Huang.ea2019}. The ultimate goal of DMD is to develop a data-driven model via an eigendecomposition of the Koopman operator, under the assumption that the full-state observable (the identity function) is in the span of the eigenfunctions \cite{SCC.Gonzalez.Abudia.ea2021}. The addition of control adds greater difficulty to data-driven methods like DMD, as the Koopman operator associated with the dynamical system depends upon the control input. Furthermore, in discrete time, the Koopman operator is generally not linear in its symbol, which makes separating the influence of the controller from the drift dynamics challenging. Despite the difficulty, there have been several successful methods for generalizing Koopman DMD for dynamical systems with control in results such as \cite{SCC.Goswami.Paley2022},\cite{SCC.Korda.Mezic2018a, SCC.Proctor.Brunton.ea2016, SCC.Williams.Hemati.ea2016, SCC.Proctor.Brunton.ea2018, SCC.Huang.Ma.ea2018}.

The method presented in \cite{SCC.Proctor.Brunton.ea2016} yields a DMD routine to represent a general nonlinear system with control as a control-affine linear system. This idea is generalized in \cite{SCC.Korda.Mezic2018a} with extended DMD (eDMD), providing greater predictive power. Furthermore, for a general discrete-time, nonlinear dynamical system with control, the authors in \cite{SCC.Korda.Mezic2018a} utilize the shift operator to describe the time evolution of the control signal. Also, in discrete-time, separation of the control input from the state can be achieved via first order approximations \cite{SCC.Straesser.Berberich.ea2023}. For continuous-time dynamical systems, the Koopman canonical transform (see \cite{SCC.Surana2016}) is used in \cite{SCC.Goswami.Paley2022} to leverage a formulation of the dynamical system in the lifted space as a control-affine, bilinear system, called the Koopman bilinear form (KBF). The KBF is then amenable to the design of feedback laws using techniques from optimal control.

The aforementioned methods demonstrate the ability to predict the response of both discrete-time \cite{SCC.Korda.Mezic2018a} and continuous-time \cite{SCC.Goswami.Paley2022} dynamical systems to open-loop inputs. The algorithm developed in this paper offers an advantage over the methods from \cite{SCC.Goswami.Paley2022, SCC.Korda.Mezic2018a, SCC.Proctor.Brunton.ea2016, SCC.Williams.Hemati.ea2016, SCC.Proctor.Brunton.ea2018, SCC.Huang.Ma.ea2018, SCC.Straesser.Berberich.ea2023}, since in addition to a predictive model, it also estimates eigenfunctions, and consequently, a Koopman invariant subspace of the \emph{closed-loop system}. A key contribution here is the extension of the method presented in \cite{SCC.Rosenfeld.Kamalapurkar2021} to the discrete-time case. The operator representation presented in \cite{SCC.Rosenfeld.Kamalapurkar2021} relies on linearity of differential and multiplication operators to separate the influence of the controlled and the uncontrolled part of the system dynamics. In discrete time, the differential operators need to be replaced by composition operators, and composition operators are typically not linear in their symbol. Herein lies the difficulty of extending continuous-time DMD results to discrete-time DMD results, as separation of the effect of the control input from the effect of the drift dynamics is nontrivial.

In this paper, we take an operator-theoretic approach to DMD with a novel operator definition that accounts for the effect of control and the discrete-time nature of the problem. The algorithm is referred to as discrete control Liouville DMD or DCLDMD for brevity. To accomplish DCLDMD, the discrete, nonlinear dynamical system is represented as a composition of two operators acting on a Hilbert space of functions. The first operator mimics the effects of a composition operator, which maps from a reproducing kernel Hilbert space (RKHS) to a vector-valued RKHS (vvRKHS). In order to account for the effect of control, we make use of a multiplication operator which maps functions in the vvRKHS back into the RKHS. In doing so, we obtain an approximate representation of the dynamical system as a composition of the aforementioned operators. 

The paper is organized into the following sections. Section \ref{sec:Background} establishes the mathematical background for dynamic mode decomposition with discrete control Liouville operators. Section \ref{sec:ProbS} contains the problem description. Section \ref{sec:DCLDMD} provides the derivation for discrete control Liouville dynamic mode decomposition, as well as outlining the DCLDMD algorithm. Section \ref{sec:NumExp} contains the numerical experiments involving the Duffing oscillator. Lastly, section \ref{sec:concl} concludes the paper.

\section{Background} \label{sec:Background}
In this section, we provide a brief overview of  RKHSs and vvRKHSs and their role in DCLDMD.

\begin{definition}\label{Def:RKHS}
An RKHS $\Tilde{H}$ over a set $X\subset \mathbb{R}^{n}$ is a Hilbert space of functions $f: X\to \mathbb{R}$ such that for all $x\in X$ the evaluation functional $E_{x}f \coloneqq f(x)$ is bounded. By the Riesz representation theorem, there exists a function $\Tilde{K}_{x} \in \Tilde{H}$ such that $f(x) ={\langle f, \Tilde{K}_{x} \rangle}_{\Tilde{H}}$ for all $f\in \Tilde{H}$. 
\end{definition}

The snapshots of a dynamical system are embedded into an RKHS via a kernel map $x \mapsto \Tilde{K}(\cdot,x) \coloneqq \Tilde{K}_{x}$. Moreover, the span of the set $\{\Tilde{K}_{x} : x \in X\}$ is dense in $\Tilde{H}$.
\begin{proposition}\label{Prop:Kdensity}
    If $A \coloneqq \{\Tilde{K}_{x} : x \in X\}$, then $\text{span }A = \Tilde{H}$.
\end{proposition}

\begin{proof}
    To show that the span of the set $\{\Tilde{K}_{x} : x \in X\}$ is dense in $\Tilde{H}$ amounts to showing that $(A^{\perp})^{\perp} = \Tilde{H}$. Let $h \in A^{\perp}$, then $\langle h, \Tilde{K}_{x}\rangle = h(x) = 0$. Hence $h\equiv 0$ on $X$. Thus $A^{\perp} = \{0\}$.
\end{proof}

In order to account for the effect of control, we make use of a vvRKHS.
\begin{definition}\label{Def:vvRKHS}
Let $\mathcal{Y}$ be a Hilbert space, and let $H$ be a Hilbert space of functions from a set $X$ to $\mathcal{Y}$. The Hilbert space $H$ is a \emph{vvRKHS} if for every $\bar{u} \in \mathcal{Y}$ and $x \in X$, the functional $f \mapsto \langle f(x), \bar{u} \rangle_{\mathcal{Y}}$ is bounded.
\end{definition}

To each $x \in X$ and $\bar{u} \in \mathcal{Y}$, we can associate a linear operator over a vvRKHS given by $(x,\bar{u}) \mapsto K_{x,\bar{u}}$, following \cite{SCC.Rosenfeld.Kamalapurkar2021}. The function $K_{x,\bar{u}}$ is known as the kernel operator and the span of these functions constitutes a dense set in the respective vvRKHS \cite[Proposition 1]{SCC.Rosenfeld.Kamalapurkar2021}. Given a function $f \in H$, the reproducing property of $K_{x,\bar{u}}$ implies ${\langle f, K_{x,\bar{u}} \rangle}_{H} = {\langle f(x), \bar{u} \rangle}_{\mathcal{Y}}$.  For more discussion on vvRKHSs see \cite{SCC.Carmeli.DeVito.ea2010}.

\section{Problem Statement} \label{sec:ProbS}
Consider a control-affine, discrete-time dynamical system of the form
\begin{equation} \label{eqn:DTDS}
x_{k+1} = F(x_{k}) + G(x_{k})u_{k},
\end{equation}
where $x \in \mathbb{R}^{n}$ is the state, $u\in\mathbb{R}^m$ is the control input, $F : \mathbb{R}^{n} \to \mathbb{R}^{n}$ and $G: \mathbb{R}^{n} \to \mathbb{R}^{n\times m}$ are functions corresponding to the drift dynamics and the control effectiveness, respectively. We refer to the individual functions which comprise the columns of $G$ by $G_{j}:\mathbb{R}^{n} \to \mathbb{R}^{n}$, for $1\leq j\leq m$. Given a feedback law $\mu :\mathbb{R}^{n} \to \mathbb{R}^{m}$ and a set of data points $\{(x_{k},x_{k+1},u_{k})\}_{k=1}^{n}$, where $u_k$ are arbitrary (potentially open-loop) control inputs, the goal is to predict the response of the system in (\ref{eqn:DTDS}) to the feedback law $\mu$.

In this paper, the set $X$ is selected to be a compact subset of $\mathbb{R}^n$, the set $\mathcal{Y}$ is selected to be $\mathbb{R}^{1\times(m+1)}$, $\Tilde{H}$ denotes an RKHS of continuous functions from $X$ to $\mathbb{R}$, and $H$ denotes a vvRKHS of continuous functions from $X$ to $\mathbb{R}^{1\times (m+1)}$.

\section{Operator representation of controlled discrete-time systems} \label{sec:DCLDMD}
A linear operator can be associated with the dynamical system in (\ref{eqn:DTDS}), as a composition of two operators: a composition-like discrete Liouville operator and a multiplication operator. This operator representation is derived in this section. 
\subsection{A Composition-like Kernel Propagation Operator}
The technical lemma below and the proposition that follows are needed for the kernel propagation operator to be well-defined.
\begin{lemma}\label{lem:perpzero}
    The set $\Omega \subset H$, defined as $\Omega  \coloneqq \{K_{x,\bar{u}} : x \in X, u\in\mathbb{R}^m, \text{ and } \bar{u}\coloneqq\begin{pmatrix}
        1 & u^{\top}
    \end{pmatrix} \in \mathcal{Y}\}$, satisfies $\Omega^\perp = \{0\}$.
\end{lemma}
\begin{proof}
    Let $h\in \Omega^\perp$. The reproducing property of $K_{x,\bar{u}}$ implies that for all $u \in \mathbb{R}^m$ and $x\in X$, $\left\langle h(x),\begin{pmatrix}
        1&u^{\top}
    \end{pmatrix}\right\rangle_{\mathcal{Y}} = \left\langle h,K_{x,\bar{u}}\right\rangle_{H}$. Since $h\in \Omega^\perp$ and $K_{x,\bar{u}} \in \Omega$, we conclude that $\left\langle h,K_{x,\bar{u}}\right\rangle_{H} = 0$. As a result, for each fixed $x \in X$ and for all $u\in\mathbb{R}^m$, we have $\left\langle h(x),\begin{pmatrix}
        1&u^{\top}
    \end{pmatrix}\right\rangle_{\mathcal{Y}} = 0$. Since the only such $ h(x) \in \mathcal{Y} $ is the zero vector, we conclude that $h = 0$.
\end{proof}
\begin{proposition}
Let $L_z \in H$ be a function such that for all tuples $(x,u,y)$ satisfying $y=F(x)+G(x)u$, we have
\begin{equation}
\left\langle [L_{z}](x), \begin{pmatrix}1&u^{\top}\end{pmatrix} \right\rangle_{\mathcal{Y}} = \Tilde{K}_{z}(y). \label{eq:operatorDef}
\end{equation} For all $z\in X$, the map $\Tilde{K}_z \mapsto L_z$ 
is a well-defined operator.
\end{proposition}

\begin{proof}
    For a given $z\in X$, suppose there are two functions, $L_z^1$ and $L_z^2$, each of which satisfy \eqref{eq:operatorDef} given above. Then, for any tuple $(x, y, u)$ which satisfies $y = F(x) + G(x)u$,
\begin{equation*}
    {\left\langle [L_z^1](x), \begin{pmatrix}1&u^{\top}\end{pmatrix} \right\rangle}_{\mathcal{Y}} = {\left\langle [L_z^2](x), \begin{pmatrix}1&u^{\top}\end{pmatrix} \right\rangle}_{\mathcal{Y}},
\end{equation*}
and therefore,$ {\left\langle [L_z^1](x) - [L_z^2](x), \begin{pmatrix}1&u^{\top}\end{pmatrix} \right\rangle}_{\mathcal{Y}} = 0$.

Using the reproducing property,
\begin{equation}\label{eqn:uno}
    {\left\langle [L_z^1] - [L_z^2], K_{x,\bar{u}} \right\rangle}_{H} = 0
\end{equation} 
for all vectors in the set $\Omega \coloneqq \{K_{x,\bar{u}} : x \in X \text{ and } \bar{u} \in \mathcal{Y} \mid \bar{u} = \begin{pmatrix}1&u^{\top}\end{pmatrix}, u \in \mathbb{R}^{m} \}$. 

As a result, $[L_z^1] - [L_z^2] \in \Omega^{\perp}$, where $\perp$ denotes the orthogonal complement of $\Omega \subset \mathcal{Y}$. 

Since $\Omega^\perp = \{0\}$ according to lemma \ref{lem:perpzero}, we conclude that for all $z\in X$,
$[L_z^1] =  [L_z^2]$.
That is, the operator $\tilde{K}_z\mapsto L_z$ is well defined on the set $\{\tilde{K}_{z}\}_{z\in X}$. Linearity of the operator then implies that it is also well-defined on $\vspan\{\tilde{K}_{z}\}_{z\in X}$.
\end{proof}

\begin{definition}\label{Def:DiscreteLiouvilleOperator}
Let $A_{F,G}:\mathcal{D}\left(A_{F,G}\right)\to H$ be the operator with domain $\mathcal{D}\left(A_{F,G}\right)\coloneqq\vspan\{\tilde{K}_z\}_{z\in X}$ that maps, for each $z\in X$, the function $\Tilde{K}_{z}$ to a function $[A_{F,G}\Tilde{K}_{z}] \in H$ such that for all tuples $(x,u,y)$ satisfying $y=F(x)+G(x)u$, we have
\begin{equation}
{\left\langle \left[A_{F,G}\Tilde{K}_{z}\right](x), \begin{pmatrix}1&u^{\top}\end{pmatrix} \right\rangle}_{\mathcal{Y}} = \Tilde{K}_{z}(y).
\end{equation}
\end{definition} 

A few remarks regarding definition \ref{Def:DiscreteLiouvilleOperator} are in order. The kernel propagation operator $A_{F,G}$ is composition-like in the sense that if a linear kernel $\tilde{K}_{z}(x)=z^{\top}x$ is used, one could define $A_{F,G}$ explicitly as $A_{F,G}\tilde{K}_{z} = [\tilde{K}_{z}(F(\cdot)),\tilde{K}_{z}(G_1(\cdot)),\cdots,\tilde{K}_{z}(G_m(\cdot))]$. In that case, due to linearity of the kernel,
$\left\langle \left[A_{F,G}\tilde{K}_{z}\right](x), \begin{pmatrix}1&u^{\top}\end{pmatrix} \right\rangle_{\mathcal{Y}}
= \tilde{K}_{z}(F(x)) + \sum_{j=1}^m \tilde{K}_{z}(G_j(x))u_j =  \tilde{K}_{z}(y)$. That is, similar to the Koopman operator for autonomous systems, the operator $A_{F,G}$, when composed with the inner product operation in a RKHS with a linear kernel, propagates the observable $\tilde{K}_{z} $ one step forward in time.

In the case of nonlinear kernels, an explicit expression for the operator $A_{F,G}$ cannot be derived. However, the implicit definition above, which achieves one-step propagation of the kernels by definition, is still useful for DMD. 

Since $\vspan\{\Tilde{K}_x\}_{x\in X}$ is dense in $H$, the kernel propagation operator $A_{F,G}$ is densely defined. As such, the adjoint $A^*_{F,G}$ exists and can be defined through its domain.
\begin{definition}\label{Def:AdjointK_FG}
The domain of the adjoint $A_{F,G}^{*}$ of $A_{F,G}$ is defined as $\mathcal{D}(A_{F,G}^{*}) \coloneqq \{  f \in H : h \mapsto \langle A_{F,G}h, f \rangle_{H} \text{ is bounded on $\mathcal{D}(A_{F,G})$} \}.$
\end{definition} 

Note that for all $x\in X$ and $\bar{u}\in \mathcal{Y}$, the kernel functions $K_{x,
\bar{u}}$ of $H$ are in the domain of the adjoint $A^{*}_{F,G}$. Indeed, if $A_{F,G}h \in H$, $ \langle A_{F,G}h, K_{x,\bar{u}} \rangle$ is bounded by definition \ref{Def:vvRKHS} and hence $K_{x,\bar{u}} \in \mathcal{D}(A_{F,G}^{*})$.


\subsection{Multiplication Operators}
Let $\nu : X \to \mathcal{Y}$ be a continuous function. The multiplication operator with symbol $\nu$ is denoted as $M_{\nu} : \mathcal{D}(M_{\nu}) \to \Tilde{H}$. For a function $h \in \mathcal{D}(M_{\nu})$, we define the action of the multiplication operator on $h$ as
\begin{equation*}
[M_{\nu}h](\cdot) = {\langle h(\cdot), \nu(\cdot)\rangle}_{\mathcal{Y}},
\end{equation*}
where the domain of the multiplication operator is given as 
\begin{equation*}
\mathcal{D}(M_{\nu}) \coloneqq \{h\in H \mid x \mapsto {\langle h(x), \nu(x)\rangle}_{\mathcal{Y}} \in \Tilde{H}\}.
\end{equation*}

For completeness, we recall the interaction between multiplication operators and kernel operators from \cite{SCC.Rosenfeld.Kamalapurkar2021}. The interaction is  used to calculate the finite-rank representation of the composition of the multiplication operator with the  kernel propagation operator from Definition \ref{Def:DiscreteLiouvilleOperator}. 
\begin{proposition}\label{Prop2}
    Suppose that $\nu : X \to \mathcal{Y}$ corresponds to a densely defined multiplication operator $M_{\nu} : \mathcal{D}(M_{\nu}) \to \Tilde{H}$ and $\Tilde{K} : X\times X \to \mathbb{R}$ is the kernel function of the RKHS $\Tilde{H}$. Then, for all $x\in X$, $\Tilde{K}_{x} \in \mathcal{D}(M^{*}_{\nu})$, where $M^{*}_{\nu}$ is the adjoint of $M_{\nu}$, and $
        M^{*}_{\nu}\Tilde{K}_{x} = K_{x,\nu(x)}$.
\end{proposition}
The composition of the  kernel propagation operator from Definition \ref{Def:DiscreteLiouvilleOperator} and the multiplication operator can be used to  define the discrete control Liouville operator.

\subsection{The Discrete Control Liouville Operator}

Taking the composition of $A_{F,G}$ and $M_{\nu}$, for a known feedback law $\mu : \mathbb{R}^{n}\to\mathbb{R}^{m}$, the evolution of an observable along trajectories of the dynamical system can be described in terms of an infinite-dimensional linear operator. 
\begin{definition}\label{Def:DCLOperator}
 Let $\nu \coloneqq \begin{pmatrix}1&\mu^{\top}\end{pmatrix} \in H$. The discrete control Liouville operator corresponding to the closed-loop system \begin{equation*}
    x_{k+1} = F(x_k) + G(x_k)\mu(x_k)
\end{equation*} is defined as the composition $M_{\nu}A_{F,G} : \mathcal{D}(M_{\nu}A_{F,G}) \to \tilde{H}$, where $\mathcal{D}(M_{\nu}A_{F,G}) = \mathcal{D}(A_{F,G})$.
\end{definition}
The discrete control Liouville operator governs the flow of observables in $\mathcal{D}(M_{\nu}A_{F,G})\subseteq\Tilde{H}$ along trajectories of the discrete-time dynamical system as $[M_{\nu}A_{F,G}h](x_{k}) = \left\langle [A_{F,G}h](x_{k}),\begin{pmatrix}1&\mu(x_{k})^{\top}\end{pmatrix} \right\rangle_{\mathcal{Y}} = h(x_{k+1})$. Furthermore, the composition $M_{\nu}A_{F,G}$ is a linear operator by linearity of the inner product and by definition \ref{Def:DiscreteLiouvilleOperator}.

\section{Discrete-time control Liouville DMD}
In order to represent the infinite-dimensional discrete control Liouville operator as a finite-dimensional operator, we select bases  $\alpha = \left\{\Tilde{K}_{x_{i}}\right\}_{i=1}^{n} \subset \Tilde{H}$ and $\beta = \left\{K_{x_{i},\bar{u}_{i}}\right\}_{i=1}^{n} \subset H$, where $\bar{u}_{i}\coloneqq \begin{pmatrix}
    1 & u_{i}^\top
\end{pmatrix} \in \mathcal{Y}$. DMD is then performed via an eigendecomposition of the finite-dimensional representation.

Given an observable $h\in \Tilde{H}$, let $\tilde{h} \coloneqq P_{\alpha}h = \sum_{i=1}^{n}\Tilde{a}_{i}\Tilde{K}_{x_{i}}$ be the projection of $h$ onto $\vspan\alpha$. One can recover a finite rank proxy of the discrete control Liovuille operator by observing its action restricted to $\vspan\alpha \subset \Tilde{H}$  and projecting the output $M_{\nu}A_{F,G}\Tilde{h}$ back onto $\vspan\alpha$. That is, recovering the finite-rank proxy amounts to writing $P_{\alpha}M_{\nu}A_{F,G}\Tilde{h}$ as $\sum_{i=1}^{n}\Tilde{b}_{i}\Tilde{K}_{x_{i}}$ and finding a matrix that relates the coefficients  $\{\Tilde{a}_{i}\}_{i=1}^{n}$ and $\{\Tilde{b}_{i}\}_{i=1}^{n}$. For brevity of notation, let $\Tilde{a} \coloneqq \begin{pmatrix}
    \Tilde{a}_{1} & \ldots & \Tilde{a}_{n}
\end{pmatrix}^\top$ and $\Tilde{b} \coloneqq \begin{pmatrix}
    \Tilde{b}_{1} & \ldots & \Tilde{b}_{n}
\end{pmatrix}^\top$. The coefficients can be computed by solving the linear system of equations (see \cite{SCC.Rosenfeld.Kamalapurkar2021} and \cite{SCC.Gonzalez.Abudia.ea2021})
\begin{equation}
\tilde{G} \begin{pmatrix}\tilde{b}_1\\ \vdots \\ \tilde{b}_{n} \end{pmatrix} = \begin{pmatrix}\langle M_{\nu}P_{\beta}A_{F,G}\tilde{h}, \tilde{K}_{x_1} \rangle_{\Tilde{H}} \\ 
\vdots \\ \langle M_{\nu}P_{\beta}A_{F,G}\tilde{h}, \tilde{K}_{x_{n}} \rangle_{\Tilde{H}}\end{pmatrix},
\end{equation}
where $\Tilde{G} = \{\Tilde{K}(x_{i},x_{j})\}_{i,j=1}^n$ is the kernel gram matrix for $\alpha$.
Since the kernel functions in $\alpha \subset \Tilde{H}$ are in the domain of the adjoint of the multiplication operator (see proposition \ref{Prop2}), for all $j$, $ \langle M_{\nu}P_{\beta}A_{F,G}\tilde{h}, \tilde{K}_{x_j}\rangle_{\Tilde{H}} = \langle A_{F,G}\tilde{h}, P_{\beta}M_{\nu}^{*}\tilde{K}_{x_j}\rangle_{H}$. Furthermore, by linearity of $A_{F,G}$,
\begin{align*}
&\langle A_{F,G}\tilde{h}, P_{\beta}M_{\nu}^{*}\tilde{K}_{x_{j}} \rangle_{H} = \sum_{i=1}^{n}\tilde{a}_i \langle A_{F,G}\tilde{K}_{x_i}, P_{\beta}M_{\nu}^{*}\tilde{K}_{x_j} \rangle_{H} \\ 
&\qquad\qquad= \sum_{i=1}^{n}\tilde{a}_i \langle A_{F,G}\tilde{K}_{x_i}, \sum_{k=1}^{n} w_{k,j}K_{x_k, \bar{u}_{k}} \rangle_{H},
\end{align*}
where $\{{w_{k,j}}\}_{k=1}^{n}$ are weights in the projection of $M_{\nu}^{*}\Tilde{K}_{x_{j}}$ onto $\vspan\beta$ and $w_{j} \coloneqq \begin{pmatrix}
    w_{1,j} & \ldots & w_{n,j}
\end{pmatrix}^\top$. Thus, 
\begin{align*}
&\langle A_{F,G}\tilde{h}, P_{\beta}M_{\nu}^{*}\tilde{K}_{x_{j}} \rangle_{H} = \sum_{i,k=1}^{n} \tilde{a}_i w_{k,j} \langle A_{F,G}\tilde{K}_{x_i}, K_{x_k,\bar{u}_{k}} \rangle_{H} \\
&\,\,= \sum_{i=1}^{n} \sum_{k=1}^{n}\tilde{a}_i w_{k,j} \langle [A_{F,G}\tilde{K}_{x_i}](x_k), \begin{pmatrix} 1&{u_{k}}^{\top} \end{pmatrix} \rangle_{\mathcal{Y}} = {\tilde{a}}^{\top} \tilde{I} w_{j}, 
\end{align*} where $\tilde{I} = \left( \left\langle [A_{F,G}\tilde{K}_{x_i}](x_k), {\begin{pmatrix} 1&{u_{k}}^{\top} \end{pmatrix}} \right\rangle_{\mathcal{Y}} \right)_{i,k=1}^{n}$  
is computed using the fact that $\left\langle [A_{F,G}\tilde{K}_{x_i}](x_k), \begin{pmatrix} 1&{u_{k}}^{\top} \end{pmatrix} \right\rangle_{\mathcal{Y}} = \tilde{K}_{x_i}(x_{k+1})$.

Since $M_{\nu}^{*}$ maps $\tilde{K}_{x_{j}}$ to $K_{x_{j}, \nu(x_j)}$, the coefficients $w_{j}$ in the projection of $K_{x_{j}, \nu(x_j)}$ onto $\text{span }\beta \subset H$ are solutions of
\begin{equation}\label{eqn6}
G \begin{pmatrix}w_{1,j}\\ \vdots \\ w_{n,j} \end{pmatrix} = \begin{pmatrix}\langle K_{x_j,\nu(x_j)}, K_{x_1,\bar{u}_{1}}\rangle_H \\ 
\vdots \\ \langle K_{x_{j},\nu(x_{j})}, K_{x_n,\bar{u}_{n}} \rangle_H\end{pmatrix},
\end{equation}
where $G = \left(\left\langle K_{x_i,\bar{u}_{i}}, K_{x_j,\bar{u}_{j}} \right\rangle_H \right)_{i,j=1}^{n}$ and $\nu(x_{j}) = \begin{pmatrix}
        1 & \mu(x_{j})^{\top}
\end{pmatrix} $. If a diagonal kernel operator $K_{x_i} \coloneqq \text{diag}\begin{pmatrix}
    \Tilde{K}_{x_1}&\ldots & \Tilde{K}_{x_{m+1}}
\end{pmatrix}$ is used, with $\Tilde{K}_{x_{j}} = \Tilde{K}_{x_{i}}$ for $1\leq j\leq m+1$, then the inner products in $G$ can be computed as \begin{multline}
    \langle K_{x_i,\bar{u}_{i}}, K_{x_j,\bar{u}_{j}} \rangle_{H}  = \langle K_{x_i,\bar{u}_{i}}(x_j), \begin{pmatrix}1&u_j^{\top}\end{pmatrix} \rangle_{\mathcal{Y}} = \\\begin{pmatrix}1&u_i^{\top}\end{pmatrix}\tilde{K}(x_j,x_i)\begin{pmatrix}1&u_{j}^{\top}\end{pmatrix}^{\top}.\label{eq:GCompute}
\end{multline}
Letting $I_{j}^{\top}$ denote the column vector on the right-hand side of (\ref{eqn6}), he $j$th row of $I$ is given by
\begin{equation*}
I_{j} = \left( {\langle K_{x_j,\nu(x_j)}, K_{x_1,\bar{u}_{1}}\rangle}_{H},\ldots,
{\langle K_{x_{j},\nu(x_{j})}, K_{x_n,\bar{u}_{n}} \rangle}_{H} \right).
\end{equation*} 

The complete finite-rank representation of the DCLDMD operator is then recovered as $[M_{\nu}P_\beta A_{F,G}]^{\alpha}_{\alpha} = \tilde{G}^{-1}IG^{-1}\tilde{I}^{\top}$, where the subscript $\alpha$ denotes the restriction of the operator to the $\vspan\alpha$, and the superscript $\alpha$ denotes projection of the output onto $\vspan\alpha$.  

\subsection{Discrete Control Liouville Dynamic Mode Decomposition}\label{subsec:DCLDMD}
DMD can be accomplished via an eigendecomposition of the finite-rank proxy of discrete control Liovuille operator. Let $\{v_{i},\lambda_{i}\}_{i=1}^{n}$ be the eigenvalue-eigenvector pairs of the matrix $[M_{\nu}P_\beta A_{F,G}]^{\alpha}_{\alpha}$. Following \cite{SCC.Gonzalez.Abudia.ea2021}, if $v_{j}$ is an eigenvector of the matrix $[M_{\nu}P_\beta A_{F,G}]^{\alpha}_{\alpha}$, then the function $\varphi_{j} = \sum_{i=1}^{n}(v_{j})_{i}\Tilde{K}_{x_{i}}$ is an eigenfunction of the operator $P_{\alpha}M_{\nu}P_{\beta}A_{F,G}\vert_{\alpha}$. 

If $\varphi_{j}$ is an eigenfunction of $P_{\alpha}M_{\nu}P_{\beta}A_{F,G}\vert_{\alpha}$ with eigenvalue $\lambda_{j}$, then
\begin{equation*}
\varphi_{j}(x_{k+1}) =  M_{\nu}A_{F,G}\varphi_{j}(x_{k}) = \lambda_{j} \varphi_{j}(x_{k}).
\end{equation*}
Hence, the eigenfunctions evolve linearly along the flow. The normalized eigenfunctions are defined as $\hat{\varphi}_{j} \coloneqq \frac{1}{\sqrt{v_{j}^{\top}\Tilde{G}v_{j}}}\sum_{i=1}^{n}(v_{j})_{i}\Tilde{K}_{x_{i}}$.

Assuming the the $j$-th component identity function, $g_{id}$, defined as $g_{\mathrm{id},j}(x)\coloneqq x_{j}$ is in $\mathcal{D}(M_{\nu}A_{F,G})\subset \Tilde{H}$, for each $j=1,2,\ldots, n$, we can describe the evolution of the full-state observable $g_{\mathrm{id}}(x)=x$ as a linear combination of eigenfunctions of $M_{\nu}A_{F,G}$.
This approach yields a data-driven model of the closed-loop dynamical system as a linear combination of eigenfunctions of the operator $P_{\alpha}M_{\nu}P_\beta A_{F,G}\vert_{\alpha}$.  
That is, for a given $x_{0} \in X$ we have a pointwise approximation of the flow of the closed-loop system
\begin{equation}
    x_{k+1} = F(x_{k})+G(x_{k})\mu(x_{k}) \approx \sum_{i=1}^{n} \lambda_{i}^{k} \xi_{i} \hat{\varphi}_{i}(x_{0}).
\end{equation}
We refer to the vectors $\xi_{i}$ as the \textit{Liouville Modes}, these are the coefficients required to represent the full-state observable as a linear combination of the eigenfunctions. We can calculate the modes by solving $g_{\mathrm{id}}(x)=x=\sum_{i=1}^{n}\xi_{i}\varphi_{i}$ for $\xi_i$, which yields $\xi \coloneqq \begin{pmatrix}\xi_1& \cdots& \xi_n\end{pmatrix} = X(V^{\top}\Tilde{G})^{-1}$, where $V$ is the matrix of normalized eigenvectors of the finite-rank representation $[M_{\nu}P_\beta A_{F,G}]^{\alpha}_{\alpha}$ and $X\coloneqq \begin{pmatrix}x_{1}&\ldots&x_{n}\end{pmatrix}$ is the data matrix. We refer to this method as the \textit{direct reconstruction} of the flow.

We can also formulate an \textit{indirect reconstruction} of the flow by considering the function $F_{\mu} \coloneqq x\mapsto \sum_{i=1}^{n} \lambda_{i} \xi_{i} \hat{\varphi_{i}}(x)$ that approximates the closed loop dynamics under the feedback law $\mu$ as $x_{k+1} \approx F_{\mu}(x_{k})$. The indirect method generally performs better for approximating the nonlinear dynamics; we hypothesize that the better performance is due to the fact we are estimating nonlinear dynamics using nonlinear functions, as the indirect reconstruction yields a nonlinear model of the flow, as opposed to the direct reconstruction which is linear. Due to its superior performance, we will use the indirect reconstruction in the numerical experiments in  section \ref{sec:NumExp}. The DCLDMD algorithm is summarized in Algorithm 1. 

\section{Convergence Properties of DCLDMD}\label{sec:Convergence}

Discrete control Liouville DMD enjoys convergence guarantees on par with current state-of-the-art Koopman methods. That is, the sequence of finite-rank operators  $P^{n}_{\alpha}M_{\nu}P^{n}_\beta A_{F,G}P^{n}_{\alpha}$, where $P^{n}$ denotes the projection onto the $n$-dimensional span of $\alpha$ and $\beta$, respectively, converges to the operator $M_{\nu}A_{F,G}$ in the \textit{strong operator topology} (SOT). Underlying this fact is the assumption that as $n\to \infty$, the Gram matrices $\Tilde{G}$ and $G$ do not become rank deficient. 
\begin{theorem}
    If $A_{F,G}: \Tilde{H} \to H$ and $M_{\nu}:\mathcal{D}(M_{\nu})\to \Tilde{H}$ are bounded, and $\alpha \coloneqq \{\Tilde{K}_{x_{n}}\}_{n=1}^{\infty} \subset \Tilde{H}$ and $\beta \coloneqq \{{K}_{x_{n},\bar{u}_{n}}\}_{n=1}^{\infty} \subset H$ are two orthonormal sequences in $\Tilde{H}$ and $H$, respectively, then for all $f\in \Tilde{H}$,  $\lim_{n\to\infty}\norm{P^{n}_{\alpha}M_{\nu}P^{n}_\beta A_{F,G}P_{\alpha}^{n}f - M_{\nu}A_{F,G}f}_{\Tilde{H}} = 0$.
\end{theorem}
\begin{proof}
     Suppose $f\in \Tilde{H}$, then
    \begin{multline*}
     \norm{P^{n}_{\alpha}M_{\nu}P^{n}_\beta A_{F,G}P_{\alpha}^{n}f - M_{\nu}A_{F,G}f}_{\Tilde{H}} = \\
     \norm{ (P^{n}_{\alpha}-I)M_{\nu}P^{n}_\beta A_{F,G}P_{\alpha}^{n}f +   M_{\nu}(P^{n}_\beta A_{F,G}P^{n}f - A_{F,G}f)}_{\Tilde{H}} \\
     \leq \norm{(P^{n}_{\alpha}-I)(M_{\nu}P^{n}_\beta A_{F,G}P_{\alpha}^{n}f - M_{\nu}A_{F,G}f)}_{\Tilde{H}} +\\ \norm{(P^{n}_{\alpha}-I)M_{\nu}A_{F,G}f}_{\Tilde{H}}  + \norm{M_{\nu}(P^{n}_\beta A_{F,G}P_{\alpha}^{n}f - A_{F,G}f)}_{\Tilde{H}} \\
     \leq \norm{(P^{n}_{\alpha}-I)}_{op}\norm{(M_{\nu}P^{n}_\beta A_{F,G}P_{\alpha}^{n}f - M_{\nu}A_{F,G}f)}_{\Tilde{H}} +\\ \norm{(P^{n}_{\alpha}-I)M_{\nu}A_{F,G}f}_{\Tilde{H}}  + \norm{M_{\nu}(P^{n}_\beta A_{F,G}P_{\alpha}^{n}f - A_{F,G}f)}_{\Tilde{H}},
    \end{multline*}
    where $\norm{\cdot}_{op}$ denotes the operator norm. Since $M_{\nu}$ is continuous and $\norm{(P^{n}_{\alpha}-I)}_{op}$ is bounded (by Parseval's identity, see \cite[Section 3.1.11]{SCC.Pedersen1989}), and since $P^{n}_\beta A_{F,G}P_{\alpha}^{n}$ converges to $A_{F,G}$ in the SOT \cite[Page 172]{SCC.Pedersen1989}), the first and the third terms in the inequality above converge to 0 as $n\to\infty$. The fact that $P^{n}_{\alpha}$ converges to $I$ in the SOT implies the convergence of the second term to zero. Therefore, the sequence of operators $P^{n}_{\alpha}M_{\nu}P^{n}_{\beta} A_{F,G}P_{\alpha}^{n}$ converges to $M_{\nu}A_{F,G}$ in the SOT.
\end{proof} 
Convergence in the SOT does not guarantee convergence of the spectrum, but by theorem 4 in \cite{SCC.Korda.Mezic2018}, it does guarantee that there is a subsequence of eigenvalue-eigenfunction pairs of the finite-rank representation which converges to an eigenvalue-eigenfunction pair of the true discrete control Liouville operator.

\section{Numerical Experiments}\label{sec:NumExp}
As a demonstration of the efficacy of the developed DCLDMD algorithm, we apply the method to the controlled Duffing oscillator and compare it with the linear predictor developed in \cite{SCC.Korda.Mezic2018a}.
\begin{experiment}\label{exp:Duffing}
The controlled Duffing oscillator is a nonlinear dynamical system with state-space form 
\begin{multline}\label{eqn: Duffing}
     \begin{pmatrix}
         \dot{x}_{1} \\
         \dot{x}_{2}
     \end{pmatrix} = \begin{pmatrix}
         x_{2} \\
         -\delta x_{2} -\beta x_{1} -\alpha x_{1}^{3}
     \end{pmatrix} +
     \begin{pmatrix}
         0 \\
         2+\sin(x_{1})
     \end{pmatrix}u
\end{multline}
where $\alpha,\beta,\delta$ are coefficients in $\mathbb{R}$, $[x_{1},x_{2}]^{\top}\in\mathbb{R}^{2}$ is the state, and $u \in \mathbb{R}$ is the control input. For the experiments the parameters are selected to be: $\delta =0$, $\alpha =1$, and $\beta =-1$.

We descretize (\ref{eqn: Duffing}) using a time step of $0.01$ seconds to yield a discrete-time, control-affine dynamical system of the form $x_{k+1} = F(x_{k}) + G(x_{k})u_{k}$. Using the tuples $\{(x_{k}, x_{k+1}, u_{k})\}_{k=1}^{n}$ generated by the dynamical system, we aim to predict the response of the system  starting from the initial condition $x_{0}=[2,-2]^{\top}$ to two different feedback laws, $\mu(x_{k})=-2x_{k,1}-x_{k,2}$ and $\bar{\mu}(x_{k})=-2x^{3}_{k,1}-x_{k,2}$ for a total of $5$ seconds.

In the implementation of DCLDMD for the linear feedback law, $\mu$, we generate 225 data points $\{(x_{k}, x_{k+1}, u_{k})\}_{k=1}^{225}$ with initial conditions sampled from a $15\times 15$ grid within the set $[-3,3]\times [-3,3]\subset \mathbb{R}^{2}$. The control inputs are sampled uniformly from the interval $[-2,2]\subset \mathbb{R}$. For the case of the nonlinear feedback law, $\Bar{\mu}$, we generate 1225 data points from initial conditions sampled from a $35\times 35$ grid within the set $[-5,5]\times [-5,5]\subset \mathbb{R}^{2}$ and the control input are sampled uniformly from the interval $[-8,8]\subset \mathbb{R}$.

\begin{algorithm}[t]
    \caption{\label{alg:DCLDMD}The DCLDMD algorithm}
    \begin{algorithmic}[1]
        \renewcommand{\algorithmicrequire}{\textbf{Input:}}
        \renewcommand{\algorithmicensure}{\textbf{Output:}}
        \REQUIRE Data points $\{(x_{k},y_{k}, u_{k})\}_{k=1}^{n}$ that satisfy $y_{k}=F(x_{k})+G(x_{k})u_{k}$, reproducing kernels $\tilde{K}_{x_{j}}$ and $K_{x_{j},\bar{u}_{j}}$ for $\Tilde{H}$ and $H$, respectively. A feedback law $\mu$, kernel parameter $\sigma$, and a regularization parameter $\epsilon$.
        \ENSURE $\{\hat{\varphi}_{j},\lambda_j, \xi_{j}\}_{j=1}^{n}$
        \STATE $\Tilde{G} \leftarrow \{\Tilde{K}(x_{i},x_{j})\}_{i,j=1}^n$ 
        \STATE $\Tilde{I} \leftarrow \{\Tilde{K}(x_{k+1},x_{i})\}_{k,i=1}^{n}$ 
        \STATE $G \leftarrow \{\langle K_{x_i,\bar{u}_i}, K_{x_j,\bar{u}_{j}} \rangle_H \}_{i,j=1}^{n}$ (see \eqref{eq:GCompute})
        \STATE $I \leftarrow \{ \langle K_{x_j,\nu(x_j)}, K_{x_i,\bar{u}_{i}}\rangle_H\}_{i,j=1}^{n}$ (see \eqref{eq:GCompute})
        \STATE Compute $[M_{\nu}P_{\beta}A_{F,G}]_{\alpha}^{\alpha} = \tilde{G}^{-1}IG^{-1}\tilde{I}^{\top}$
        \STATE Eigendecomposition: $\{\varphi_j,\lambda_j\}_{j=1}^{n} \leftarrow [M_{\nu}P_\beta A_{F,G}]^{\alpha}_{\alpha}$
        \STATE Normalize the eigenfunctions: $\{\hat{\varphi}_{j}\}_{j=1}^{n} \leftarrow \hat{\varphi}_{j} = \frac{1}{\sqrt{v_{j}^{\top}\Tilde{G}v_{j}}}\sum_{i=1}^{n}(v_{j})_{i}\Tilde{K}_{x_{i}}$
        \STATE Liouville modes: $\xi \leftarrow X(V^{\top}\Tilde{G})^{-1}$
        \RETURN $\{\hat{\varphi}_{j}, \lambda_j, \xi_{j}\}_{j=1}^{n}$
    \end{algorithmic} 
\end{algorithm}

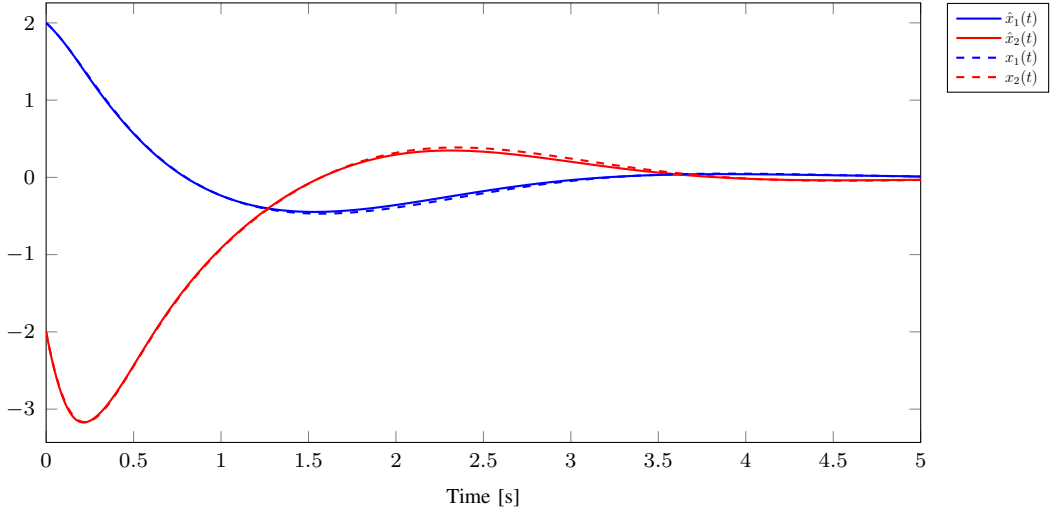
\begin{figure}
   \centering
    \begin{tikzpicture}
    \begin{axis}[
        xlabel={Time [s]},
        ylabel={},
        legend pos = outer north east,
        legend style={nodes={scale=0.5, transform shape}},
        enlarge y limits=0.05,
        enlarge x limits=0,
        height = 0.45\columnwidth,
        width = 0.8\columnwidth,
        label style={font=\scriptsize},
        tick label style={font=\scriptsize}]
        \addplot [thick, blue] table [x index=0, y index=1]{data/DuffingDiscreteReconComp.dat};
        \addplot [thick, red] table [x index=0, y index=2]{data/DuffingDiscreteReconComp.dat};
        \addplot [thick, blue, dashed] table [x index=0, y index=3]{data/DuffingDiscreteReconComp.dat};
        \addplot [thick, red, dashed] table [x index=0, y index=4]{data/DuffingDiscreteReconComp.dat};
         \legend{$\hat{x}_{1}(t)$,$\hat{x}_{2}(t)$, $x_{1}(t)$,$x_{2}(t)$}
    \end{axis}
\end{tikzpicture}
    \caption{A comparison of indirectly reconstructed trajectories $\hat{x}_{1}(t)$ and $\hat{x}_{2}(t)$ with the true trajectories $x_{1}(t)$ and $x_{2}(t)$ of the Duffing oscillator resulting from the linear feedback law $\mu$ in experiment \ref{exp:Duffing}.}
    \label{fig:DuffingDiscreteReconComp}
\end{figure}

In both implementations of DCLDMD, the Gaussian  radial basis function kernel $\Tilde{K}(x,y) = \mathrm{e}^{\frac{-{\left\Vert x-y\right\Vert}_{2}^2}{\sigma}}$ is used for calculation of the Gram matrices associated with $\alpha \subset \Tilde{H}$. The kernel width is set to $\sigma=10$ and $\sigma=20$ for the response of the system to $\mu$ and $\Bar{\mu}$, respectively. For $\beta \subset H$, we associate to each pair $\{(x_{k},u_{k})\}_{k=1}^{n}$ a kernel $K_{x_{k},\bar{u}_k} \coloneqq \begin{pmatrix}
    1&u_k^\top
\end{pmatrix}K_{x_{k}} \in H$. Here we use the kernel operator $K_{x_i} \coloneqq \text{diag}\begin{pmatrix}
\Tilde{K}_{x_1} & \cdots & \Tilde{K}_{x_{m+1}}
\end{pmatrix}$ where $\Tilde{K}_{x_j}(y) = \mathrm{e}^{\frac{-{\left\Vert x_{j}-y\right\Vert}_{2}^2}{\sigma}}$ for $j=1,\ldots,m+1$.  
Lastly, we select $\varepsilon = 10^{-6}$ for regularization of the Gram matrices in order to ensure invertibility of both $\Tilde{G}$ and $G$ in the finite-rank representation (see Algorithm \ref{alg:DCLDMD}).

A comparison between the true trajectories and the indirectly reconstructed trajectories corresponding to the feedback laws $\mu$ and $\Bar{\mu}$ can be seen in figures \ref{fig:DuffingDiscreteReconComp} and \ref{fig:DuffingDiscreteNLReconComp}, respectively.


\begin{figure}
  \centering
   \begin{tikzpicture}
    \begin{axis}[
        xlabel={Time [s]},
        ylabel={},
        legend pos = outer north east,
        legend style={nodes={scale=0.5, transform shape}},
        enlarge y limits=0.05,
        enlarge x limits=0,
        height = 0.45\columnwidth,
        width = 0.8\columnwidth,
        label style={font=\scriptsize},
        tick label style={font=\scriptsize}
    ]
        \addplot [thick, blue] table [x index=0, y index=1]{data/NLDuffingDiscreteReconComp.dat};
        \addplot [thick, red] table [x index=0, y index=2]{data/NLDuffingDiscreteReconComp.dat};
        \addplot [thick, blue, dashed] table [x index=0, y index=3]{data/NLDuffingDiscreteReconComp.dat};
        \addplot [thick, red, dashed] table [x index=0, y index=4]{data/NLDuffingDiscreteReconComp.dat};
         \legend{$\hat{x}_{1}(t)$,$\hat{x}_{2}(t)$, $x_{1}(t)$,$x_{2}(t)$}
    \end{axis}
\end{tikzpicture}
    \caption{A comparison of indirectly reconstructed trajectories $\hat{x}_{1}(t)$ and $\hat{x}_{2}(t)$ with the true trajectories $x_{1}(t)$ and $x_{2}(t)$ of the Duffing oscillator resulting from the nonlinear feedback law $\Bar{\mu}$ in experiment \ref{exp:Duffing}.}
    \label{fig:DuffingDiscreteNLReconComp}
\end{figure}
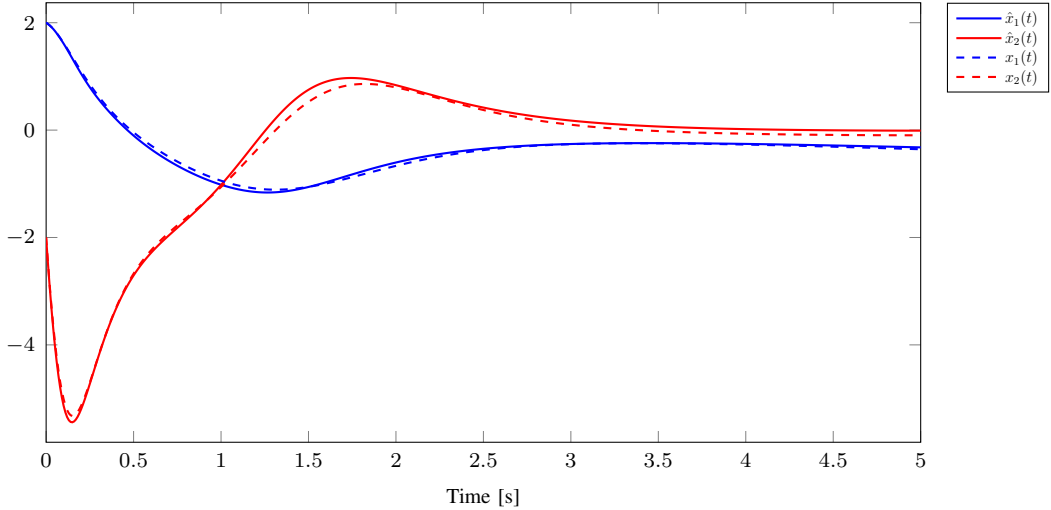
\end{experiment}

\begin{experiment}\label{exp:PredictionComp}
In this experiment, we compare the predictive capabilities of the indirect reconstruction via DCLDMD with the linear predictor derived in \cite{SCC.Korda.Mezic2018a}. The linear predictor in \cite{SCC.Korda.Mezic2018a} is of the form $z_{k+1} = Az_{k} + Bu_{k}$ with $x_{k} = Cz_{k}$ and $z$ being the lifted state (see \cite{SCC.Korda.Mezic2018a} for more details). For a given feedback law $\mu$, we can estimate the response of the Duffing oscillator described by equation (\ref{eqn: Duffing}) to the feedback law $\mu$ by using the linear predictor $z_{k+1} = Az_{k} + B\mu(Cz_{k})$.

For this experiment, we generate 1000 data points and DCLDMD is implemented using the same kernels as in experiment \ref{exp:Duffing}, except the kernel widths are both set to $\sigma = 100$. For regularization we set $\varepsilon = 10^{-6}$. For the linear predictor, extended DMD (eDMD) is performed with the Gaussian radial basis functions as in \cite{SCC.Korda.Mezic2018a}. For the initial condition $x_{0} = [2, -2]^{\top}$ and the feedback law $\mu(x_{k})=-2x_{k,1}-2x_{k,2}$, we compare the predictions of the indirect DCLDMD method and the linear predictor with the true trajectories (see figure \ref{fig:PredictorComp}).
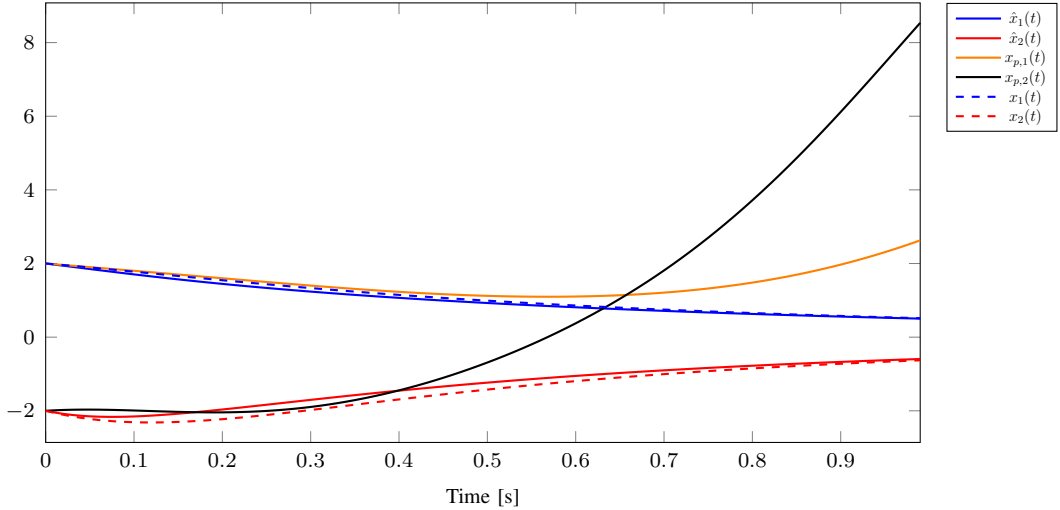
\begin{figure}
   \centering
             \begin{tikzpicture}
    \begin{axis}[
        xlabel={Time [s]},
        ylabel={},
        legend pos = outer north east,
        legend style={nodes={scale=0.5, transform shape}},
        enlarge y limits=0.05,
        enlarge x limits=0,
        height = 0.45\columnwidth,
        width = 0.8\columnwidth,
        label style={font=\scriptsize},
        tick label style={font=\scriptsize}
    ]
        \addplot [thick, blue] table [x index=0, y index=1]{data/LinvsDCLDMD_DiscreteReconCompare.dat};
        \addplot [thick, red] table [x index=0, y index=2]{data/LinvsDCLDMD_DiscreteReconCompare.dat};
        
        \addplot [thick, orange] table [x index=0, y index=3]{data/LinvsDCLDMD_DiscreteReconCompare.dat};
        \addplot [thick, black] table [x index=0, y index=4]{data/LinvsDCLDMD_DiscreteReconCompare.dat};
        
        \addplot [thick, blue, dashed] table [x index=0, y index=5]{data/LinvsDCLDMD_DiscreteReconCompare.dat};
        \addplot [thick, red, dashed] table [x index=0, y index=6]{data/LinvsDCLDMD_DiscreteReconCompare.dat};
         \legend{$\hat{x}_{1}(t)$, $\hat{x}_{2}(t)$, $x_{p,1}(t)$, $x_{p,2}(t)$, $x_{1}(t)$, $x_{2}(t)$}
    \end{axis}
\end{tikzpicture}
    \caption{A comparison between the linear predictor developed in \cite{SCC.Korda.Mezic2018a} and the indirect reconstruction via DCLDMD in experiment \ref{exp:PredictionComp}. Here, $\hat{x}_{i}(t)$, $x_{p,i}(t)$, and $x_{i}(t)$ represent the indirect reconstruction, the linear predictor, and the actual trajectories, respectively, where $i$ is a subscript denoting an element of the state.}
    \label{fig:PredictorComp}
\end{figure}
\end{experiment}

\subsection{Discussion}
The experiments demonstrate the efficacy of DCLDMD in an academic setting with the Duffing oscillator. The experiments are done with no prior model knowledge, besides the system being affine in control. The novelty of the representation can be seen in the separation of the control input and the state on the operator-theoretic level, while still preserving the nonlinearity of the dynamical system. This is opposed to the standard approach for discrete-time dynamical systems where the lifted state $z_{k} \in \mathbb{R}^{N}$ can be approximated as $z_{k+1}\approx Az_{k}+Bu_{k}$, with $A \in \mathbb{R}^{N\times N}$ and $B\in \mathbb{R}^{1\times N}$ found using extended DMD. Unless the original nonlinear system admits an exact lifting, which is not generally the case, the trajectories of the linear lifted systems  are expected to diverge from the trajectories of the nonlinear system with increasing prediction horizons.


In experiment \ref{exp:PredictionComp}, specifically, in figure \ref{fig:PredictorComp}, we observe that as expected, the behavior of the linear predictor from \cite{SCC.Korda.Mezic2018a} diverges from the behavior of the nonlinear Duffing oscillator under the given feedback law, while the indirect reconstruction approach developed in this paper accurately tracks the actual trajectory of the Duffing oscillator. We postulate that the improved predictive capability can be attributed to the fact that the indirect predictive model is a \textit{nonlinear} predictor, as opposed to the model from \cite{SCC.Korda.Mezic2018a}, which is a linear predictor, albeit in a higher dimensional lifted state space.


In both experiments, indirect reconstruction is used to estimate the flow. The indirect reconstruction explicitly depends upon the eigenfunctions of $P_{\alpha}M_{\nu}P_{\beta}A_{F,G}\vert_{\alpha}$. Whether or not we can always represent the full-state observable (i.e. the flow) in terms of the eigenfunctions is not entirely clear, but this is a standard assumption in the DMD literature. With this assumption in mind, DCLDMD is termed a heuristic approach for estimation of the dynamics. Regardless, the numerical experiments in section \ref{sec:NumExp} demonstrate the capability of DCLDMD for prediction of the response of the control-affine system to given feedback laws.

\section{Conclusion} \label{sec:concl}
In this paper, a novel operator representation of a control-affine nonlinear system  is developed as a composition of a multiplication operator and a composition-like kernel propagation operator over an RKHS. The multiplication operator takes advantage of the affine nature of the system to capture the effect of control on the system behavior, while the kernel-propagation operator captures the effect of the system dynamics on the kernels of the underlying RKHS. The resulting DMD algorithm is entirely data driven and requires no model knowledge besides the dynamical system being affine in control. Furthermore, the DCLDMD formulation provides a novel way to separate the state from the control input on the operator-theoretic level. This separation leads to better prediction capabilities over existing methods, as evidenced by the results of Experiment 2. Moreover, since DCLDMD can be used to predict closed-loop trajectories of a nonlinear system under feedback laws, it could potentially be utilized for control synthesis, which is a topic for future research.\vspace{-1em}

\bibliographystyle{ieeeTRAN}
\bibliography{scc, sccmaster,DCLDMDbibliography,DCLDMDTemp,temp}

\end{document}